\documentclass[11pt]{article}
\usepackage[numbers]{natbib}
\usepackage{graphicx,amsmath,epsfig,amssymb,amsfonts,amsthm,enumerate,verbatim,todonotes,hyperref,mathtools, dsfont}

\usepackage{tikz,tikz-cd}
\usetikzlibrary{matrix,arrows}

\newtheorem{theorem}{Theorem}[section]
\newtheorem{corollary}[theorem]{Corollary}
\newtheorem{lemma}[theorem]{Lemma}

\newtheorem{question}[theorem]{Question}
\newtheorem{claim}[theorem]{Claim}

\theoremstyle{definition}
\newtheorem{definition}[theorem]{Definition}

\newcommand{\restrict}{\mathord{\upharpoonright}}

\begin{document}

\title{Splitting Localization and Prediction Numbers}
\author{Iv\'an Ongay-Valverde\footnote{Work done while being supported by CONACYT scholarship for Mexican student studying abroad.}\\\emph{Department of Mathematics}\\\emph{University of Wisconsin--Madison}\\\emph{Email: ongay@math.wisc.edu}
}

\date{\begin{tabular}{rl}First Draft:&September 15, 2017\\Current Draft:&August 7, 2019\end{tabular}}
\maketitle
\begin{abstract}
In this paper the work done by Newelski and Roslanowski in \cite{newelski1993ideal} is revisited to solve a question posed by Blass about one of the possible evasion and prediction numbers (see \cite{blass}). This led to define a variation of the $k$-localization property (the $(k+1)^{\omega}$-localization property) and the use of a forcing notion with accelerating trees. 
\end{abstract}

\section{Introduction}

In 1993 Newelski and Roslanowski defined the $k$-localization number, $\mathcal{L}_{k}$ (see \cite{newelski1993ideal}), as the minimal cardinality of a family $\mathcal{T}$ of $k$-trees such that every element $(k+1)^{\omega}$ is a branch of a tree in $\mathcal{T}$.\footnote{In their work, they originally studied ideals of unsymmetric games. The covering numbers of those ideals are the ones that we called $k$-localization numbers.}

In their paper, they proved that $\mathcal{L}_{k+1}\leq \mathcal{L}_{k}$ and that it is consistent to have $\mathcal{L}_{k+1}<\mathcal{L}_{k}$. In order to do this, they introduce the $k$-localization property\footnote{The $k$-localization property says ``all the reals in $\omega^{\omega}$ of the generic extension are a branch of a $k$-tree from the ground model.''} that was later studied by Roslanowski \cite{roslanowski2006n} and Zapletal \cite{zapletal2008n}. These properties were also used by Geschke \cite{geschke2002convexity} to show that it is consistent to have $\mathcal{L}_{i}=f(i)$ for any non-increasing function from a natural number to the cardinals with uncountable cofinality.

Nevertheless, the $k$-localization property is not the minimum necessary to have $\mathcal{L}_{k}^{V}=\mathcal{L}_{k}^{V[G]}$. The minimum that we need to have would be the property ``all the reals in $(k+1)^{\omega}$ of the generic extension are a branch of a $k$-tree from the ground model" (we call this the $(k+1)^{\omega}$-localization property). Does this mean that there is a cardinal characteristic that is closer to the $k$-localization property than $\mathcal{L}_{k}$?

There is one. In his chapter of the Handbook of Set Theory \cite{blass}, Andreas Blass talks about cardinal characteristics related to the concepts of evasion and prediction. At the end of that section, he introduces 36 variations of these cardinals and left as an open question to pin down 4 of them whose identity didn't appear to be one of the known cardinal characteristic. It turns out that the same proof of Newelski and Roslanowski shows that one of these variations, specifically the prediction number for global adaptive $k$ predictors, is not one of the known cardinal characteristics and, actually, gives countable many cardinal characteristics which, consistently, can take different values (see Theorem \ref{Prediction Hierarchy}).

This triggers the following question: are the variation of prediction and the $k$-localization number equal? This paper shows that they are not. It is consistent to have all the prediction numbers mention above at value $\aleph_{2}=\mathfrak{c}$ and all localization numbers at value $\aleph_{1}$ (see Theorem \ref{main theorem}).

To do this we use a forcing that Noah Schweber and the author called accelerating tree forcing. We created it for a computability theory question in a coauthor paper still in preparation\footnote{It is possible to find an early version of it at http://www.math.wisc.edu/~ongay/publications.html} and we will show in this paper that countable support product of the accelerating tree forcing has the $3^{\omega}$-localization property.

It has been pointed out to us that the accelerating tree forcing could be related to bushy tree forcing (as done in \cite{khan2017forcing}) or other fast-growing tree forcing (as done in \cite{ciesielski1999model}). Although we got inspiration from them, the fact that we allowed long stretches with no split makes us believe that this forcing is conceptually of a different kind.
	
It is also important to remark that countable support iteration and product of forcings with the $(k+1)^{\omega}$-localization property could also have the $(k+1)^{\omega}$-localization property, as the $k$-localization property (see \cite{zapletal2008n}). Notice that these two properties are in the same line as the Sacks property. It is unknown to the author if there is a bigger theory or theorem that handle all of them at once. This, we believe, is an interesting topic.

About the structure of the paper, it has a first section with definition and background. Then, the bulk of the work is done in Section \ref{Lemmas section}, were the lemmas involving the $(k+1)^{\omega}$ localization property are shown. Section \ref{Main reusl section} has the main theorem (Theorem \ref{main theorem}) with some conclusions and open problems. The last section was added after the paper was first submitted, it includes results that the author discover after sharing the work here presented.

Finally, the author wants to thank Noah Schweber for his support and for convincing me about publishing this work; Arnold Miller, for showing me a new way to order my thoughts, and Kenneth Kunen for all the advice and guidance with this project and others.

Also, special thanks to the referee who help refine the paper overall and give important suggestions to improve the presentation of Theorem \ref{iteration lemma}. 

\section{Definitions and background}

These first definitions will be useful during the paper:

\begin{definition}
\begin{enumerate}

\item We say that $T\subseteq \omega^{<\omega}$ is a tree if and only if given $\sigma\in T$ we have that $\sigma\restrict j\in T$ for all $j<|\sigma|$.

\item A $k$-branching tree, is a tree such that every node has either $1$ successor or $k$ of them.

\item A $k$-tree is a tree such that every node has at least $1$ successor and no more than $k$.
\end{enumerate}
\end{definition}

Now, the following definition is due to \cite{newelski1993ideal} (they express it as the covering number of an ideal):

\begin{definition}
The $k$-localization number, $\mathfrak{L}_{k}$, is the smallest cardinality of a family of $k$-branching trees that cover $(k+1)^{\omega}$.
\end{definition}

Notice that the definition is not trivial for $k\geq 2$. Furthermore, Newelski and Roslanowski showed in \cite{newelski1993ideal} that, for $k\geq 2$, $\mathfrak{L}_{k}\geq \max\{cov(\mathcal{M}), cov(\mathcal{N})\}$, that $\mathfrak{L}_{k+1}\leq\mathfrak{L}_{k}$  and that it is consistent that $\mathfrak{L}_{k+1}<\mathfrak{L}_{k}$.

On the other hand, in Blass's chapter of the Handbook of Set Theory \cite{blass} he defines:

\begin{definition}
\begin{enumerate}
\item A $k$ globally adaptive predictor is a sequence of functions $\pi=\langle \pi_{n}:n\in \omega\rangle$ with $\pi_{n}:\omega^{n}\rightarrow [\omega]^{k}$. We say that a function $f\in \omega^{\omega}$ is predicted by $\pi$ if there is $m\in\omega$ such that for all $n>m$, $f(n)\in \pi_{n}(f\restrict n)$.

\item The $k$ globally prediction number, $\mathfrak{v}_{k}^{g}$, is the minimal cardinality of a set of $k$ globally adaptive predictors that predict all functions in $\omega^{\omega}$.

\item The $k$ globally evasion number, $\mathfrak{e}_{k}^{g}$, is the minimal cardinality of a set of functions in $\omega^{\omega}$ such that the whole set is not predicted by a single $k$ globally adaptive predictor.
\end{enumerate}
\end{definition}

It is important to make some remarks about the last definition: 
\begin{itemize}
\item The `adaptive' part refers to the fact that $\pi_{n}$ is not constant. Non-adaptive objects are closer to slaloms (or traces).

\item The `globally' part of the definition refers to the fact that we have $\pi_{n}$ for all $n\in \omega$. It is possible to define predictors using $\pi_{n}$ for $n\in D\subsetneq \omega$. 

\item Blass do not give a notation for this number, so the notation $\mathfrak{v}_{k}^{g}$ and $\mathfrak{e}_{k}^{g}$ is introduced here.

\item These definitions are not trivial for $k\geq 2$.

\end{itemize}

The numbers $\mathfrak{v}_{k}^{g}$ and $\mathfrak{e}_{k}^{g}$ are duals between them and, by the work done in \cite{blass}, we know that $\mathfrak{m}_{\sigma-\mbox{centered}}\leq \mathfrak{e}_{k}^{g} \leq \mbox{add}(\mathcal{N})$. So, by duality, we know that $\mbox{cof}(\mathcal{N})\leq \mathfrak{v}_{k}^{g}\leq \mathfrak{c}$. Also, from the definition, we have that $\mathfrak{v}_{k+1}^{g}\leq \mathfrak{v}_{k}^{g}$.

Reading the definition more carefully we can notice that all the functions that are predicted by a $k$-globally adaptive predictor are covered by $\aleph_{0}$ many $k$-branching trees, so $\mathfrak{v}_{k}^{g}$ is also the minimum cardinal of a set of $k$-branching trees (or $k$-trees) that cover $\omega^{\omega}$.

Furthermore, in \cite{newelski1993ideal}, we have the following result:

\begin{theorem}\label{Prediction Hierarchy}
Given $k\geq 2$, it is consistent to have ZFC+$cof(\mathcal{N})= \mathfrak{v}_{k+1}^{g}<\mathfrak{v}_{k}^{g}=\mathfrak{c}$.
\end{theorem}

This theorem is a corollary of the proof of:

\begin{theorem}[Newelski, Roslanowski \cite{newelski1993ideal}]
Given $k\geq 2$, it is consistent to have ZFC+$cof(\mathcal{N})= \mathfrak{L}_{k+1}<\mathfrak{L}_{k}=\mathfrak{c}$.
\end{theorem}

Specifically, it comes from two facts: first, that the forcings that were used have the $k$-localization property. This is that ``every real in $\omega^{\omega}$ is a branch of a $k$-tree from the ground model'', this keeps $\mathfrak{L}_{k+1}$ and $\mathfrak{v}_{k+1}^{g}$ at $\aleph_{1}$; and the forcing adds a function in $(k+1)^{\omega}$ that is not the branch of any $k$-tree from the ground model. Notice that this function is also a function in $\omega^{\omega}$  that is not the branch of any $k$-tree from the ground model. Once you take a countable support product, this makes $\mathfrak{L}_{k}$ and $\mathfrak{v}_{k}^{g}$ of size $\mathfrak{c}$.

The relation between these two cardinal characteristics is more evident once we realize, from the tree definition of $\mathfrak{L}_{k}$ and $\mathfrak{v}_{k}^{g}$, that $\mathfrak{L}_{k}\leq \mathfrak{v}_{k}^{g}$. So, a natural question arises of whether $\mathfrak{L}_{k}=\mathfrak{v}_{k}^{g}$. The goal of this paper is to answer the question in a negative way.

\section{Combinatorial lemmas and localization properties}\label{Lemmas section}

In this setting it is better to understand some of the processes as combinatorial principles instead of parts of a forcing argument. Because of that, the following lemmas come in pairs: one is a combinatorial statement and the following one is the forcing result.

\begin{lemma}\label{comb case 1}
Given $\{f_{i}: i\in I\}\subseteq 3^{\omega}$ with $|I|=3^{n}$ you can find $S\subseteq I$ with $|S|=n$ such that $\{f_{i}\restrict n : i\in S, n\in\omega\}$ is a $2$-tree.
\end{lemma}

\begin{proof}
We will do the proof by induction.
 
For $n=0$ and $n=1$ it is trivially true.

Now, assume that it is true for $n$, we will prove it for $n+1$.

Given $\{f_{i}: i\in I\}\subseteq 3^{\omega}$ with $|I|=3^{n+1}$ if all of them are the same function then take the first $n+1$ of them, they make trivially a $2$-tree. On the other hand, if there are two of them that are different, find the first natural number $m$ such that two of them differ. Notice that, using a pigeon hole principle, there is a value $k\in 3$ such that there is $J\subseteq I$, with $|J|\geq 3^{n}$ such that for all $i\in J$ we have $f_{i}(m)=k$.

Now, take $i_{0}\in I$ such that $f_{i_{0}}(m)\neq k$ and let $S'\subseteq J$ be the index set of size $n$ given after using the induction hypothesis over $J$. Notice that $\{f_{i_{0}}\restrict j: j\in \omega\}\cup\{f_{i}\restrict j: i\in S', j\in \omega\}$ forms a $2$-tree and that $S=S'\cup\{i_{0}\}$ has size $n+1$.

\end{proof}

\begin{lemma}\label{ground lemma}
There is a forcing notion that adds a function from $\omega$ to $\omega$ that is not predicted by any $k$-global adaptive predictor but such that all reals in $3^{\omega}$ are a branch of a $2$-tree in the ground model.
\end{lemma}

\begin{proof}

\begin{definition}
We say that $T\subseteq \displaystyle \bigcup_{m\in \omega}\prod_{n\in m}^{m}(n+1)$ is an accelerating tree if and only if it is a subtree of $\displaystyle \bigcup_{m\in \omega}\prod_{n\in m}^{m}(n+1)$, if every node has an extension that splits and given $\sigma\in T$ such that there are $k_{i}\in \omega$, $i<n$, such that $\sigma\restrict k_{i}$ is a splitting node (i.e., $\sigma$ has $n$ splits before it) then $\sigma$ has either $1$ successor or  at least $n+2$.\footnote{During the conference 'Set Theory of the Reals', BIRS-CMO Oaxaca, August 2019, it was brought to our attention that this forcing was originally defined by Geschke in \cite{GeschkeDual} as Miller Lite Forcing.}
\end{definition}

Let $\mathds{P}$ be the forcing notion whose conditions are of the form $\langle \tau, T\rangle$ with $\tau\in \displaystyle \bigcup_{m\in \omega}\prod_{n\in m}^{m}(n+1)$ and $T$ an accelerating subtree of $\displaystyle \bigcup_{m\in \omega}\prod_{n\in m}^{m}(n+1)$ extending $\tau$. We say that $\langle \tau', T'\rangle\leq \langle \tau, T\rangle$ if and only if $\tau\subseteq \tau'$, $T'\subseteq T$ and $\tau'\in T$.\footnote{We decide to define the acceleration tree forcing using pairs to create a stronger resemblance to the effective analogue of accelerating trees of $\omega^{\omega}$ (paper in preparation with Noah Schweber). Furthermore, this will allow us to easily define $(T)^{0}$ in Lemma \ref{iteration lemma}. }

For a node $\rho\in T$, let \[T_{\rho}=\{\tau\in T : \tau\subseteq \rho\ \vee \ \rho\subseteq \tau\}\]

Notice that given any $k$-tree $U\subseteq \omega^{<\omega}$ and  a condition $\langle \tau, T\rangle$, there is $\rho\in T$ that is not a node in $U$ (for example, go to a split with $k+1$ nodes, one of them is not in $U$). Furthermore, if we take the condition $\langle \rho, T_{\rho} \rangle$, none of the branches of $T_{\rho}$ are branches of $U$. This shows that forcing with accelerating tree forcings adds a function from $\omega$ to $\omega$ that is not predicted by any $k$-global adaptive predictor.

Now, we will only give a sketch on how to proof that all reals in $3^{\omega}$ are a branch of a $2$-tree in the ground model. This part of the theorem is a corollary of Lemma \ref{iteration lemma} letting $\kappa=1$. Also, the forcing is the set theoretical version of the forcing used in a paper to appear with Noah Schweber\footnote{You can find an early version of it at http://www.math.wisc.edu/~ongay/publications.html .}. From that proof, translating from computability theory to set theory, we have the desired result.

Now, the sketch: the idea is that given a condition $p\in\mathds{P}$ and a $\mathds{P}$-name $\dot f$ such that $p\Vdash \dot f\in 3^{\omega} $ we can define, in $V,$ a $2$-tree, $A$, and a condition $q\leq p$ such that $q\Vdash \dot f\in A $. To do this, you can prune the tree of $p$, call it $T$, in such a way that given $\sigma\in T$, an $n$-th split node of $T$, there is $g_{\sigma}\in 3^{n} $ such that $\langle \sigma, T_{\sigma}\rangle\Vdash \dot f\restrict n= g_{\sigma}$.

We are looking for $T'\subseteq T$ such that $\{g_{\sigma} : \sigma\in T' \mbox{ $n$-split of $T$, } n\in \omega\} $ is inside a ground model $2$-tree we do the following process: if we already decided that $\sigma\in T'$ then we look for an extension of $\sigma$ with $3^{n}$ successors, say $\sigma_{i}$, $i\in 3^{n}$. Now we look for $M\in \omega$, $\rho_i$ extending $\sigma_{i}$ and $g_i\in 3^{M}$ such that $\langle \rho_i, T_{\rho_{i}}\rangle\Vdash \dot f\restrict M= g_{i}$. Now, running the above combinatorial Lemma \ref{comb case 1}, select $S\subseteq 3^{n}$ such that $\{g_i\restrict j: i\in S, j<M\}$ is a $2$-tree (definable in $V$ using the definability lemma of forcing).  Finally, we ensure that there is no spliting between $\sigma_{i}$ and $\rho_{i}$ in T' and we ask that $\sigma_{i}, \rho_{i}\in T'$ if and only if $i\in S$.
\end{proof}

\begin{lemma}\label{comb iteration}
Given $\{f_{i}^{j}: i\in I, j\in k\}\subseteq 3^{\omega}$ with $k\in \omega$, $|I|=N(n,k)$ a big enough number and $m\in \omega$ that makes $\{f_{i}^{j}\restrict l: i\in I, j\in k, l\in m+1\}$ a $2$-tree such that if $f_{i}^{j}\restrict m=f_{s}^{t}\restrict m$ with $t\neq j$ we have that $f_{i}^{j}=f_{s}^{t} $ then you can find $S\subseteq I$ with $|S|=n$ such that $\{f_{i}^{j}\restrict l: i\in S, j\in k, l\in \omega\}$ is a $2$-tree.
\end{lemma}

\begin{proof}
We will prove this by induction over $k$.

At $k=1$, we need $N(n,1)\geq 3^{n}$ so that we can use Lemma \ref{comb case 1} to be done.

Now, assuming we have the case for $k$ we will prove it for $k+1$. We need $N(n,k+1)\geq 3^{N(n,k)}$, with this we can use Lemma \ref{comb case 1} over $\{f_{i}^{k}: i\in I\}$ to get $J\subseteq I$ such that $|J|=N(n,k)$ and $\{f_{i}^{t}\restrict l: i\in J, t= k, l\in \omega\}$ is a $2$-tree. Now, we can use our induction hypothesis over $\{f_{i}^{t}: i\in J, t\in k\}$ to get $S\subseteq J$ of size $n$ such that $\{f_{i}^{j}\restrict l: i\in S, j\in k, l\in \omega\}$ is a $2$-tree.

We just need to show that \[\{f_{i}^{j}\restrict l: i\in S, j\in k+1, l\in \omega\}=\{f_{i}^{k}\restrict l: i\in S, l\in \omega\}\cup \{f_{i}^{j}\restrict l: i\in S, j\in k, l\in \omega\}\] is a $2$-tree.

Assume that we have $a\in \omega$ and $\langle i, j\rangle, \langle s, t\rangle, \langle g, h\rangle\in S\times (k+1)$  different between them such that $f^{j}_{i}\restrict a=f^{t}_{s}\restrict a=f^{h}_{g}\restrict a$. We have to show that \[|\{f^{j}_{i}\restrict (a+1), f^{t}_{s}\restrict (a+1), f^{h}_{g}\restrict (a+1)\}|\leq 2.\]

Taking into account that $\{f_{i}^{j}\restrict l: i\in I, j\in k+1, l\in m+1\}$, $\{f_{i}^{k}\restrict l: i\in S, l\in \omega\}$ and $\{f_{i}^{j}\restrict l: i\in S, j\in k, l\in \omega\}$ are $2$-trees, the only case left to check is when $a\geq m$ and $j$, $t$ and $h$ are not all the same but at least one of them is equal to $k$. Without lost of generality, assume that $h=k$ and $j\neq k$.

Since $a\geq m$ we have that $f^{k}_{g}\restrict m=f^{j}_{i}\restrict m$. Using the fact that $j\neq k$, and our theorem's hypothesis, we have that $f^{k}_{g}=f^{j}_{i}$, so \[\{|f^{j}_{i}\restrict (a+1), f^{t}_{s}\restrict (a+1), f^{h}_{g}\restrict (a+1)\}|\leq 2.\]

\end{proof}

It is important to remark that in these combinatorial lemmas it is never used that the domain of the functions is $\omega$, so these lemmas are also true for $3^{n}$. 

\begin{definition}
A forcing notion has the $(k+1)^{\omega}$ localization property if and only every function in $(k+1)^{\omega}$ in the generic extension is a branch of a $k$-tree from the ground model.
\end{definition}

\begin{lemma}\label{iteration lemma}
Countable product of the accelerating tree forcing has the $3^{\omega}$ localization property.
\end{lemma}

Newelski and Roslanowski, in \cite{newelski1993ideal}, define the $k$-localization property as the fact that all branches of $\omega^{\omega}$ are cover by a $k$-tree of the ground model. This property was deeply study later by Roslanowski, in \cite{roslanowski2006n}, and by Zapletal, in \cite{zapletal2008n}. They found that the $k$-localization property is preserved under most of the used countable support product and iteration of proper forcings.

Our forcing does not have the $2$-localization property, it will have a version of that for $3^{\omega}$: the $3^{\omega}$ localization property. Our proof will resemble the one did by Newelski and Roslanowski, nevertheless, it is possible that there are results in the lines of the other two papers.

\begin{proof}

First, given a tree and $n>0$, we let $(T)^{n}$ be the set of all nodes such that they are the successors of the $n$-th split. As a convention, given $p=\langle s, T\rangle$ a forcing condition, we have that $(T)^{0}=\{s\}$. Now, given elements of the accelerating tree forcing we will define for $n\geq 1$, $p=\langle s, T\rangle\leq_{n} p'=\langle s', T'\rangle$ if and only if $\langle s, T\rangle\leq \langle s', T'\rangle$ and $(T')^{k}=(T)^{k}$, for all $1\leq k \leq n$, and $p\leq_{0} p'$ if and only if $p\leq p'$. Notice that, since these are subtrees of $\displaystyle \bigcup_{m\in \omega}\prod_{n\in m}^{m}(n+1)$, these orders have the fusion property and satisfy Axiom A (as in \cite{baumgartner1983iterated}).

Assume that we have a countable support product of the accelerating tree forcing of length $\kappa$. Call the final partial order $\mathds{P}_{\kappa}$, as notation we will express $q\in \mathds{P}_{\kappa}$ as $q=\langle r, T\rangle$ and $q(\alpha)=\langle r(\alpha), T(\alpha)\rangle$.

Given $F\in [\kappa]^{<\omega}$ and $\eta:F\rightarrow \omega$, we define $p\leq_{F, \eta}q$ if and only if $p\leq q$ and for all $\alpha\in F$ we have that $p(\alpha)\leq_{\eta(\alpha)}q(\alpha)$. Furthermore, given $\sigma\in \prod_{\alpha\in F}T(\alpha)$ and $p\in \mathds{P}_{\kappa}$ we define $p\ast \sigma$ to be $p(\beta)$ if $\beta\notin F$ and $p(\beta)\ast \sigma(\beta)=\langle \sigma(\beta), T_{\sigma(\beta)} \rangle$ if $\beta\in F$ (following the notation of Lemma \ref{ground lemma}).

The orders $\leq_{F, \eta}$ have the fusion property under the following conditions: given $p_{n+1}\leq_{F_{n}, \eta_{n}} p_{n}$ with $\bigcup_{n\in \omega}F_{n}=\bigcup_{n\in \omega} supp(p_{n})$ and $\lim_{n\rightarrow \infty}\eta_{n}(\alpha)=\infty$ for all $\alpha\in \bigcup_{n\in \omega}F_{n}$ we have that there exist $q\in \mathds{P}_{\kappa}$ such that $q\leq_{F_{n}, \eta_{n}} p_{n}$ for all $n\in \omega$.

In order to complete the proof, it is enough to define the following concept and show the following claim:

\begin{definition}
Given $\Vdash_{\mathds{P}} \mbox{``} \dot f \in 3^{\omega}\mbox{''}$. We say that the 5-tuple $\langle q, F, \eta, m, A \rangle$ \textit{consolidates} $\dot f$ if and only if the following is satisfied:
\begin{enumerate}
\item $q=\langle r, T \rangle\in \mathds{P}_{\kappa}$, $F\in [\kappa]^{<\omega}$, $\eta:F\rightarrow \omega$, $m\in \omega$.
\item $A\subseteq 3^{<m+1}$ is a $2$-tree, $q\Vdash ``\dot f\restrict m \in A"$.
\item For each $\sigma\in \prod_{\alpha\in F}(T(\alpha))^{\eta(\alpha)}$ there is $g\in A$ such that $q\ast \sigma\Vdash ``\dot f\restrict m=g\mbox{''}$.
\item If there are a condition $q^{\ast}\leq _{F, \eta} q$, $M\in \omega$, $h\in 3^{M}$ and $\sigma_{1}\neq \sigma_{2}\in \prod_{\alpha\in F}(T(\alpha))^{\eta(\alpha)}$ such that $q^{\ast}\ast \sigma_{1}\Vdash ``\dot f\restrict M= h \mbox{''}$ but $q^{\ast}\ast \sigma_{2}\Vdash ``\dot f\restrict M\neq h\mbox{''}$ then there is $g\in A$ such that $q\ast\sigma_{1}\Vdash ``\dot f\restrict m=g\mbox{''}$ and $q\ast\sigma_{2}\Vdash ``\dot f\restrict m\neq g\mbox{''}$.
\end{enumerate}

\end{definition}

\begin{claim}\label{iteration claim}
Working in $V$, suppose that $\Vdash_{\mathds{P}} \mbox{``} \dot f \in 3^{\omega}\mbox{''}$ and $\langle q, F, \eta, m, A\rangle$ that consolidates $\dot f$. Then there are $M'>m$, $A'\subset 3^{<M'+1}$ a $2$-tree with $A= A'\cap 3^{<m+1}$ and $q'=\langle r', T'\rangle \leq_{F, \eta}q$ such that $\langle q', F, \eta+1, M, A'\rangle$ also consolidates $f$.
\end{claim}

If we prove this claim, given $p\in \mathds{P}_{\kappa}$ such that $p\Vdash ``\dot f\in 3^{\omega} "$ we can define $q_{n}$, $F_{n}$, $\eta_{n}$, $A_{n}$, $m_{n}$ as follows:

\begin{enumerate}
\item $q_{0}=p$, $A_{0}=\{\emptyset\}$ and $m_0=0$.
\item We write $supp(q_{0})=\{\alpha_{0}^{i}: i\in \omega\}$ and let $F_{0}=\{\alpha_{0}^{0}\}$.
\item We let $\eta_{0}(\alpha_{0}^{0})=0$. Clearly, $\langle q_{0}, F_{0}, \eta_{0}, m_{0}, A_{0}\rangle$ consolidates $\dot f$.
\item We define $q_{n+1}$, $A_{n+1}$ and $m_{n+1}$ as the result of the claim using $q_{n}$, $A_{n}$, $F_{n}$, $\eta_{n}$ and $m_{n}$.
\item We write $supp(q_{n+1})=\{\alpha_{n+1}^{i}: i\in \omega\}$ and let $F_{n+1}=F_{n}\cup\{\alpha_{i_{n}}^{j_{n}}\}$ with $\langle i_{n}, j_{n}\rangle$ following the usual enumeration of $\omega\times \omega$.
\item Finally, we let $\eta_{n+1}(\alpha)=\eta_{n}(\alpha)+1$ for $\alpha\in F_{n}$ and $\eta_{n+1}(\alpha_{i_{n}}^{j_{n}})=0$. Again, notice that $\langle q_{n+1}, F_{n+1}, \eta_{n+1}, m_{n+1}, A_{n+1}\rangle$ consolidates $\dot f$.
\end{enumerate}

With this, we can use the fusion property with $q_{n+1}\leq_{F_{n}, \eta_{n}} q_{n}$ and get $q\in \mathds{P}_{\kappa}$ such that $q\leq_{F_{n}, \eta_{n}} q_{n}$ for all $n$ so we have that $q\Vdash \dot ``\dot f\in [\bigcup_{n\in \omega}A_{n}]\mbox{''}$.

This shows that all the functions in $3^{\omega}$ in the extension are a branch of a ground model $2$-tree.

It is important to notice that the properties gave to the $2$-tree in the above definition and claim aligns with those in the hypothesis of Lemma \ref{comb iteration}. We will use this lemma in the proof. In order to do that, we need a couple of observations and reductions.

Below we assume that  $\langle q, F, \eta, m, A\rangle$  consolidates $\dot f$ and we fix $\beta\in F$. Let $\nu:F\rightarrow \omega$ such that $\nu(\alpha)=\eta(\alpha)$ if $\alpha\neq \beta$ and $\nu(\beta)=\eta(\beta)+1$. To show Claim \ref{iteration claim} we will look for $q'\leq_{F, \eta} q$ such that $ \langle q', F, \nu, M', A'\rangle$ consolidates $\dot f$ (instead of $ \langle q', F, \eta+1, M', A'\rangle$). This is enough since, changing the $\beta$ we are using, we can go from $\eta$ to $\eta+1$ using $|F|$ intermediate $\nu$ functions.

Let $n=\nu(\beta)+2$ and we let $k=\left|\prod_{\alpha\in F} (T(\alpha))^{\eta(\alpha)}\right|$. We will also use $N(n,k)$ as defined in Lemma \ref{comb iteration}.

Notice that, pruning the trees of $q$ if necessary, we can find $p_0=\langle r_{0}, T_{0}\rangle\leq_{F, \eta}q$ such that for each $t\in (T_0(\beta))^{\eta(\beta)}=(T(\beta))^{\eta(\beta)}$ we have that \[|\{s\in (T_0(\beta))^{\nu(\beta)}: s \mbox{ extends } t\}|=|\{s\in (T_0(\beta))^{\eta(\beta)+1}: s \mbox{ extends } t\}|=N(n,k).\] 
It is important to remark that, in general, $p_0\not\leq_{F,\nu} q$.

\textbf{Observation A} \emph{Suppose $M\in \omega$ and $\sigma\in \prod_{\alpha\in F} (T_0(\alpha))^{\nu(\alpha)}$. Then there exists $q^{\ast}\in \mathds{P}_{\kappa}$ such that $q^{\ast}\leq_{F, \nu} p_0$ and $q^{\ast}\ast \sigma$ forces a value to $\dot f\restrict M$.}

\begin{proof} Standard, see Lemma 1.7 of \cite{baumgartner1985sacks} \end{proof}

\textbf{Observation B} \emph{For $M\in \omega$ there exists there exists $q^{\ast}\in \mathds{P}_{\kappa}$,  $q^{\ast}\leq_{F, \nu} p_0$, such that for every $\sigma\in \prod_{\alpha\in F} (T_0(\alpha))^{\nu(\alpha)}$ we have that $q^{\ast}\ast \sigma$ forces a value to $\dot f\restrict M$.}

\begin{proof} Standard, see Corollary 1.10 of \cite{baumgartner1985sacks} \end{proof}

Now, given $p\leq_{F, \nu} p_0$ and $M> m$ we define
\[z_{p,M}=|\{a\in 3^{M}: \exists \sigma\in \prod_{\alpha\in F} (T(\alpha))^{\nu(\alpha)} (p\ast \sigma\Vdash \dot f\restrict M= a) \}|.\] Notice that $z_{p,M}\leq k\cdot N(n,k)$. Therefore, we can find  $p^{+}=\langle r^{+}, T^{+}\rangle \leq_{F, \eta} p_0$ and $M'> m$ such that $z_{p^{+}, M'}$ has maximum value.

Passing to a $\leq_{F, \nu}$ condition we may also demand that for every $\sigma\in \prod_{\alpha\in F} (T^{+}(\alpha))^{\nu(\alpha)}=\prod_{\alpha\in F} (T_0(\alpha))^{\nu(\alpha)}$, the condition $p^{+}\ast \sigma$ forces a value to $\dot f\restrict M'$.

\textbf{Observation C} \emph{Suppose that $\sigma_0, \sigma_{1}\in \prod_{\alpha\in F} (T^{+}(\alpha))^{\nu(\alpha)}$, $M''\geq M'$, $p\leq_{F, \nu} p^{+}$ and $a_{0}, a_{1}\in 3^{M''}$. If $a_{0}\neq a_{1}$ and \[p \ast \sigma_{0}\Vdash \dot f\restrict M''= a_0 \mbox{ and } p\ast \sigma_{1}\Vdash \dot f\restrict M''= a_1\] then $a_0\restrict M'\neq a_1\restrict M'$.}

\begin{proof}Suppose towards a contradiction that $a_0\restrict M'\neq a_1\restrict M'$. We can find $p''\leq_{F, \nu} p \leq_{F, \nu} p^{+}\leq_{F, \nu} p_{0}$ such that for every $\sigma\in \prod_{\alpha\in F} (T^{+}(\alpha))^{\nu(\alpha)}$, the condition $p''\ast \sigma$ forces a value to $\dot f\restrict M''$. Then, $z_{p'',M''}>z_{p^{+}, M'}$, a contradiction.\end{proof}

For every $\sigma \in \prod_{\alpha\in F}(T^{+}(\alpha))^{\eta(\alpha)}=\prod_{\alpha\in F} (T(\alpha))^{\eta(\alpha)}$ there are $N(n,k)$ many $\rho^{\sigma}\in \prod_{\alpha\in F}(T^{+}(\alpha))^{\nu(\alpha)}=\prod_{\alpha\in F} (T_0(\alpha))^{\nu(\alpha)}$ such that for all $\alpha\in F\setminus\{\beta\}$ we have that $\sigma(\alpha)=\rho^{\sigma}(\alpha)$. Fix an enumeration of these $\rho$ and define $\sigma^{\frown} i=\rho^{\sigma}_{i}$ (if we have $\sigma_{1}(\beta)=\sigma_{2}(\beta)$ then $\rho^{\sigma_{1}}_{i}=\rho^{\sigma_{2}}_{i}$ for all $i$).

Given $\sigma \in \prod_{\alpha\in F}(T^{+}(\alpha))^{\eta(\alpha)}$ and $i\in N(n,k)$, define $f_{i}^{\sigma}\in 3^{M'}$ to be such that $p^{+}\ast \sigma^{\frown}i \Vdash ``\dot f\restrict M'=f_{i}^{\sigma}\mbox{''}$. Since  $\langle q, F, \eta, m, A\rangle$ consolidates $\dot f$, we have that $\{f_{i}^{\sigma}\restrict l: i\in N(n,k), t\in \prod_{\alpha\in F}(T^{+}(\alpha))^{\eta(\alpha)}, l\in m+1\}\subseteq A$ is a $2$-tree such that if $f_{i}^{\sigma_{1}}\restrict m=f_{j}^{\sigma_{2}}\restrict m$ with $\sigma_{1}\neq \sigma_{2}$ we have that $f_{i}^{\sigma_{1}}=f_{s}^{\sigma_{2}}$.

Furthermore, using observation C, we have that given $\sigma_{1}, \sigma_{2}$ elements of $\prod_{\alpha\in F}(T^{+}(\alpha))^{\eta(\alpha)}$, $i,j\in N(n,k)$ with $\sigma_{1}\neq \sigma_{2}$ we have that if $f_{i}^{\sigma_{1}}\restrict M'=f_{j}^{\sigma_{2}}\restrict M'$ with $\sigma_{1}\neq \sigma_{2}$ we have that $f_{i}^{\sigma_{1}}=f_{s}^{\sigma_{2}}$. With this, any $2$-tree that comes from $\{f_{i}^{\sigma}: i\in N(n,k), \sigma\in\prod_{\alpha\in F}(T^{+}(\alpha))^{\eta(\alpha)}\}$ will satisfy the requirements of the claim.

Now we can use Lemma \ref{comb iteration} on $\{f_{i}^{\sigma}: i\in N(n,k), \sigma\in\prod_{\alpha\in F}(T^{+}(\alpha))^{\eta(\alpha)}\}$ so we can find $S\subseteq N(n,k)$ of size $n$ such that \[\{f_{i}^{\sigma}: i\in S, \sigma\in\prod_{\alpha\in F}(T^{+}(\alpha))^{\eta(\alpha)}\}\] is a $2$-tree.

To complete the claim, we use:
\begin{itemize}
    \item $M'$,
    \item $A'=\{f_{i}^{t}: i\in S, t\in k\}\cup A$ and
    \item $q'\leq_{F,\eta} q$ define as $q'(\alpha)=p^{+}(\alpha)$ for all $\alpha\neq \beta$ and $q'(\beta)=\langle r'(\beta), T'(\beta)\rangle $ where $r'(\beta)=r^{+}(\beta)$ and $T'(\beta)$ is an accelerating subtree of $T^{+}(\beta)$ such that \[(T'(\beta))^{\nu(\beta)}=\{\rho^{\sigma}_{i}(\beta): \sigma\in\prod_{\alpha\in F}(T^{+}(\alpha))^{\eta(\alpha)}, i\in S \}.\]\end{itemize}\end{proof}

\begin{corollary}
Countable support product of accelerating tree forcing has the $(k+1)^{\omega}$ localization property for all $k\geq 2$.
\end{corollary}

\begin{proof}
To prove this, it is enough to show that the $(k+1)^{\omega}$ localization property is implied by the $(s+1)^{\omega}$ localization property for $k\geq s\geq 2$, then, the result is a corollary of Lemma \ref{iteration lemma}.

Fix a surjective function $f:(k+1)\rightarrow (s+1)$. Notice that this function induces a surjective function $f^{\ast}:(k+1)^{\omega}\rightarrow (s+1)^{\omega}$. Now, working in a generic extension given a $s$-tree $T$ from the ground model, $(f^{\ast})^{-1}[T]$ is a $k$-tree from the ground model.

Therefore, if in the generic extension $(s+1)^{\omega}$ is covered by $s$-trees from the ground model, then $(k+1)^{\omega}$ is covered by $k$-trees from the ground model.

\end{proof}

Now, the following definition can let us expand our last result a little more.

\begin{definition}
Forcing with $k$-branching trees of $k^{<\omega}$ is the forcing notion that uses subtrees of $k^{<\omega}$ such that every node has either $1$ or $k$ successors.
\end{definition}

This forcing is used in \cite{newelski1993ideal} where Newelski and Roslanowski showed that this forcing has the $k$-localization property, i.e., that every function of $\omega^{\omega}$ in the generic extension is the branch of a $k$-tree from the ground model. Notice that this property implies the $(k+1)^{\omega}$ localization property. A first step in order to investigate if the countable support products of forcings with the $(k+1)^{\omega}$ localization property still has the $(k+1)^{\omega}$ localization is true for a bigger spectrum of forcings than the accelerating tree forcing is to show the following lemmas, that are analogues of Lemma \ref{comb iteration} and \ref{iteration lemma}:

\begin{lemma}\label{comb iteration 2}
Given $\{f_{i}^{j}: i\in I, j\in a\}\subseteq (k+1)^{\omega}$ with $a\in \omega$, $|I|=N(n,a)$ a big enough number and $m\in \omega$ that makes $\{f_{i}^{j}\restrict l: i\in I, j\in a, l\in m+1\}$ a $k$-tree such that if $f_{i}^{j}\restrict m=f_{s}^{t}\restrict m$ with $t\neq j$ we have that $f_{i}^{j}=f_{s}^{t} $ then you can find $S\subseteq I$ with $|S|=n$ such that $\{f_{i}^{j}: i\in S, j\in l\}$ is a $k$-tree.
\end{lemma}

\begin{proof}
This follows from the proofs of Lemma \ref{comb case 1} and Lemma \ref{comb iteration}, in those lemmas we had $k=2$. The same reasoning will give us this lemma.
\end{proof}

\begin{lemma}\label{iteration lemma 2}
Countable support product of alternating accelerating tree forcing and forcing with $k$-branching trees of $k^{<\omega}$ has the $(k+1)^{\omega}$ localization property.
\end{lemma}

\begin{proof}

Notice that the orders $\leq_{n}$ also make sense when forcing with $k$-branching trees of $k^{\omega}$.

The proof in full detail will have the same extension as the proof of Lemma \ref{iteration lemma}. Nevertheless, here we give a sketch of how to combine the technique used in \cite{newelski1993ideal} and the proof of \ref{iteration lemma}.

Everything works the same changing 2 for $k$ and $3$ for $k+1$. Now, to show the analogue of Claim \ref{iteration claim} we will have two cases:

\begin{enumerate}
\item If you are extending a node that comes from an  accelerating tree, then use Lemma \ref{iteration lemma 2} instead of Lemma \ref{iteration lemma}. Everything else works the same.

\item If you are extending a node that comes from a $k$-branching tree instead of using Lemma \ref{iteration lemma 2}, it is enough to find a condition like $p^{+}$.  Since the next split only has $k$ successors, they naturally form a $k$-tree. Everything else works the same as the proof of Claim \ref{iteration claim} or you can use the technique used in \cite{newelski1993ideal}.
\end{enumerate}

\end{proof}

\section{Main Theorem, conclusion and open questions}\label{Main reusl section}

\begin{theorem}\label{main theorem}
It is consistent with $ZFC$ that $\forall k\geq 2 (\mathfrak{L}_{k}<\mathfrak{v}_{k}^{g}=\mathfrak{c})$.
\end{theorem}

\begin{proof}
Starting with a model of $ZFC+GCH$ we can make a countable support product of the accelerating tree forcing describe in Lemma \ref{ground lemma}. Using Axiom A, as in in \cite{baumgartner1983iterated}, we know that the product preserves cardinals and that $\mathfrak{c}=\aleph_{2}$. Also, by Lemma \ref{iteration lemma}, the resulting model will have $\mathfrak{L}_{k}=\mathfrak{L}_{2}=\aleph_{1}$. We just need to show that in the extension $\mathfrak{v}_{k}^{g}=\aleph_{2}=\mathfrak{c}$.

Let $\mathds{P}_{\omega_{2}}=\prod_{\alpha\in \omega_{2}}\mathds{Q}_{\alpha}$ be the countable support product of accelerating tree forcings. Let $G=\{c_{\alpha} : \alpha\in \omega_{2}\}$ be generic over $\mathds{P}_{\omega_{2}}$. Now, for all $\beta<\omega_{1}$ let $T_{\beta}\subseteq \omega^{\omega}$ be a $k(\beta)$-tree, with $k(\beta)\in \omega$, in $V[G]$.

Now, in $V$, we can find $\dot T(\beta)$ a $\mathds{P}_{\alpha(\beta)}$-name for some $\alpha(\beta)\in \omega_{2}$. So, there is $\gamma\in \omega_{2}$ such that $\alpha(\beta)<\gamma$ for all $\beta$. Therefore, we have that $T_{\beta}\in V[\{c_{\alpha}: \alpha<\gamma\}]$ for all $\beta\in \omega_{1}$. 

Notice that if $T$ is an accelerating tree in $V$, then at the split $k(\beta)+1$ it has a branch that is not in $T_{\beta}$ in $V[\{c_{\alpha}: \alpha<\gamma\}]$. Then avoiding $T_{\beta}$ is a dense condition (in $V[\{c_{\alpha}: \alpha<\gamma\}]$) for accelerating trees from $V$.

Since $c_{\gamma}$ is a $V$-accelerating forcing generic over $V[\{c_{\alpha}: \alpha<\gamma\}]$, then $c_{\gamma}$ is not a branch of any $k$-tree in $V[\{c_{\alpha}: \alpha<\gamma\}]$, $k\in \omega$. Therefore, $c_{\gamma}$ is not a branch of any $T_{\beta}$.

This shows that, in $V[G]$, $\omega^{\omega}$ is not cover by $\{T_{\beta}:\beta\in \omega_{1}\}$. Since this was an arbitrary collection we have that $\mathfrak{v}^{g}_{k}=\aleph_{2}$ for all $k\in \omega$.

\end{proof}

This theorem proves that it is consistent that $\mathfrak{v}^{g}_{k}\neq \mathcal{L}_{k}$ and answers the question from Blass about the identity of $\mathfrak{v}^{g}_{k}$: they indeed are a different cardinal characteristic from the ones that are known.

Furthermore, we can see that there are more ways to do this split:

\begin{theorem}\label{main theorem 2}
For all $s\geq 2$ it is consistent with $ZFC$ that $\forall k\geq 2 (\mathfrak{L}_{s+1}<\mathfrak{L}_{s}=\mathfrak{v}_{k}^{g}=\mathfrak{c})$.
\end{theorem}

\begin{proof}

Following the same strategy as above, starting with a model of $ZFC+GCH$ we can make a countable support product of the accelerating tree forcing alternated with forcing with $s+1$-branching trees of $(s+1)^{\omega}$. Just as before, we know that the product preserves cardinals and that $\mathfrak{c}=\aleph_{2}$. Also, by Lemma \ref{iteration lemma 2}, the resulting model will have $\mathfrak{L}_{s+1}=\aleph_{1}$. We just need to show that, in the extension, $\mathcal{L}_{s}=\mathfrak{v}_{k}^{g}=\aleph_{2}=\mathfrak{c}$.

Let $\mathds{P}_{\omega_{2}}=\prod_{\alpha\in \omega_{2}}\mathds{Q}_{\alpha}$ be the countable support product of accelerating tree forcings, when $\alpha$ is even and forcing with $s+1$ subtrees of $(s+1)^{\omega}$ when $\alpha$ is odd. Let $G=\{c_{\alpha} : \alpha\in \omega_{2}\}$ be generic over $\mathds{P}_{\omega_{2}}$. 

To see that $\mathfrak{v}_{k}^{g}=\aleph_{2}=\mathfrak{c}$, we can do the same as above. Now, showing that $\mathcal{L}_{s}=\aleph_{2}=\mathfrak{c}$ can be found in \cite{newelski1993ideal}. Nevertheless, for convenience to the reader, we give an argument here: 

For all $\beta<\omega_{1}$ let $T_{\beta}\subseteq (s+1)^{\omega}$ be a $s$-tree in $V[G]$. In $V$, we can find $\dot T(\beta)$ a $\mathds{P}_{\alpha(\beta)}$-name for some $\alpha(\beta)\in \omega_{2}$. So, there is $\gamma\in \omega_{2}$ such that $\alpha(\beta)<2\cdot\gamma+1$ for all $\beta$. Therefore, we have that $T_{\beta}\in V[\{c_{\alpha}: \alpha<2\cdot \gamma+1\}]$ for all $\beta\in \omega_{1}$. 

Since $c_{2\cdot \gamma+1}$ is a generic for the forcing using $s+1$-branching trees (from $V$) of $(s+1)^{\omega}$ over $V[\{c_{\alpha}: \alpha<\gamma\}]$, then $c_{2\cdot \gamma+1}$ is not a branch of any $s$-tree in $V[\{c_{\alpha}: \alpha<\gamma\}]$ (same reasoning as in Theorem \ref{main theorem}). Therefore, $c_{2\cdot \gamma+1}$ is not a branch of any $T_{\beta}$.

This shows that, in $V[G]$, $(s+1)^{\omega}$ is not cover by $\{T_{\beta}:\beta\in \omega_{1}\}$. Since this was an arbitrary collection we have that $\mathfrak{v}^{g}_{k}=\aleph_{2}$ for all $k\in \omega$.

\end{proof}

\begin{question}\label{corey}
What is the value of $cof(\mathcal{N})$ in the above models?
\end{question}

Theorem \ref{main theorem 2} shows that in order to have different values for $\mathfrak{v}_{k}^{g}$ and $\mathcal{L}_{k}$ it is not necessary that every $\mathcal{L}_{s}$ have the same value. In this same venue, we can wonder if it is necessary that all $\mathfrak{v}^{g}_{s}$ have the same value. In other words:

\begin{question}\label{other cut}
Can we have $\mathcal{L}_{k}=\mathfrak{v}^{g}_{3}<\mathfrak{v}^{g}_{2}$ for all $k\geq 2$?
\end{question}

For this question, it is not possible to use neither the accelerating tree forcing nor the forcing with $3$-branching trees of $3^{\omega}$, which are the two ways used in this paper to make $\mathfrak{v}^{g}_{2}=\mathfrak{c}$. Another approach will be to use trees of $\omega^{\omega}$ that branches more than $3$ times at each split. Nevertheless, I do not see any good reason for that forcing to have the $3^{\omega}$-localization property. Maybe a modification of it can do the trick.

Now, during the paper, the $(k+1)^{\omega}$ localization property played a really important role. In order to show that it was preserved the proofs showed above are really case specific. This is useful for our purposes, but a question arises:

\begin{question}
Can we show that the $(k+1)^{\omega}$-localization property is preserved under countable support iteration and  products?
\end{question}

This is likely to be possible. In \cite{zapletal2008n}, Zapletal showed that the $n$-localization property is preserved under countable support product and iteration of a broad variety of forcings (some kind of definable proper forcings). 

Finally, notice that $\mathfrak{v}^{g}_{k}$ is a cardinal characteristic that is usually really closed to $\mathfrak{c}$. This is not true in cardinal arithmetic, but it is true in the Chicon Diagram: all of these numbers are above $cof(\mathcal{N})$. So, in order to work with them, it is important to use forcing notions that are tame somehow (they cannot add Cohen or random reals, for example). In this case, we used a forcing notion with the $(k+1)^{\omega}$ localization property but, in the literature, there are examples of properties like the Sacks property, the $n$-localization property and, most recently, the shrink wrapping property (see \cite{gonzalez2017sacks}) that are also tame with reals. It is important to notice that most of these `tameness' properties relates to the idea of keeping the new reals inside a tree of some sort.

\begin{question}
Is there an underlying theorem (or meta theorem) that relates all (of some) of this tameness properties? 
\end{question}

One possible result could be that all of them are preserved under countable support product of a variety of forcings, but I do not have any good guess of whether this is possible or not.

\section{Aftermath}

Time after this paper was sent for review, we were able to proved a couple of results answering the questions that appear above. Since this are small results, we decided to include them here.

This first results answer question \ref{corey}, they are the result of a conversation with Corey Switzer during the XIX Graduate Student Conference in Logic in Madison, Wisconsin, April 2018.  

\begin{lemma}
The accelerating tree forcing has the Sacks property.
\end{lemma}

\begin{proof}

\begin{definition}
Given a function $f:\omega\rightarrow \omega\setminus\{0\}$ an slalom of growth $f$ is a function $s:\omega\rightarrow [\omega]^{<\omega}$ such that $|s(n)|\leq f(n)$ for all $n$. We say that $g\in \omega^{\omega}$ goes thorugh $s$ if and only if $g(n)\in s(n)$ for all $n$.
\end{definition}

\begin{definition}\cite{shelah1998proper}
We say that a forcing has the Sacks property if and only if there is $g\in V$ such that $g:\omega\rightarrow \omega\setminus\{0\}$ and diverges to infinity such that for all $f\in \omega^{\omega}\cap V[G]$ there is a tree $T\in V$ such that $f$ is a branch of $T$ and the $n$-th level of $T$ has size $g(n)$.
\end{definition}

Notice that, given an slalom of growth $f$, we can generate a tree $T$ such that its $n$-th level has size $\prod_{i=0}^{n}f(i)$.

To show that the accelerating tree forcing has the Sacks property we will show that every real in $\omega^{\omega}\cap V[G]$ goes through an slalom $s\in V$ such that $|s(n)|\leq n!$.

Let $\mathds{P}$ be the accelerating tree forcing. From Lemma \ref{iteration lemma} we know that given a name $\dot f$ such that $\Vdash_{\mathds{P}}\dot f\in \omega^{\omega}$ then there is a condition $\langle p, T\rangle \in \mathds{P}$ such that given $\sigma\in T^{n}$ (notation defined in Lemma \ref{iteration lemma}) we have that there is $\tau_{\sigma}\in \omega^{n}$ such that $\langle \sigma, T_{\sigma}\rangle \Vdash \dot f\restrict n=\tau_{\sigma}$.

In $V$ define $s:\omega\rightarrow [\omega]^{<\omega}$ such that \[s(n)=\{\tau_{\sigma}(n): \sigma\in T^{n}\}.\]

Since $T$ is accelerating, we have that $|s(n)|\leq n!$.
\end{proof}

It is important to mention that in August 2019 it was brought to our attention that, in \cite{GeschkeDual}, Geschke showed indirectly that the accelerating tree forcing has the Sacks property\footnote{He forced the Dual Coloring Axiom and showed that this axiom implies $cov(\mathcal{N})=\aleph_{1}$. }. We hope that this more direct proof is both more convinient for the reader and, maybe, useful for future work.

\begin{theorem}
In the forcing extension generated after forcing with countable support product of accelerating tree forcing  $cof(\mathcal{N})=\aleph_{1}$.
\end{theorem}

\begin{proof}

\begin{theorem}(from \cite{blass})
$cof(\mathcal{L})$ is the cardinality of the smallest family $\mathcal{F}$  of slaloms of growth $f$ (for $f:\omega\rightarrow \omega\setminus\{0\}$ increasing and diverging to infinity) such that all reals in $\omega^{\omega}$ go through a slalom in $\mathcal{F}$.
\end{theorem}

\begin{theorem} (from \cite{shelah1998proper})
The countable support product of forcings that have the Sacks property have the Sacks property.
\end{theorem}

Notice that, if a forcing has the Sacks property then $cof(\mathcal{N})^{V}=cof(\mathcal{N})^{V[G]}$.

Since the accelerating tree forcing has the Sacks property, this shows that the model generated in Theorem \ref{main theorem} satisfies $cof(\mathcal{N})^{V[G]}=\aleph_{1}$.

In Theorem \ref{main theorem 2} we get the same using the fact that forcing with $k$-branching trees also has the Sacks property.
\end{proof}

Finally, these are answers for different variations of question \ref{other cut}.

In the flavor of the contructibility degrees\footnote{As in the paper \emph{The Cichon Diagram for Degrees of Constructibility} by Corey Switzer that can be at https://coreyswitzer.files.wordpress.com/2018/09/the-cichon-diagram-for-degrees-of-relative-constructibility.pdf}, we have that:

\begin{lemma}
If a forcing has the $k+1$ localization property but it does not have the $k$-localization then it do not have the $(k+1)^{\omega}$ localization property. 
\end{lemma}

\begin{proof}
In $V[G]$ let $b\in \omega^{\omega}$ be such that it is not in any $k$ tree from the ground model. Now, let $T$ be a $k+1$-branching tree from $V$ such that $b\in T$. Notice that, in $V$, there is a bijection $f:T\rightarrow (k+1)^{<\omega}$ so, in $V[G]$, this induces a function $f^{\ast}:[T]\rightarrow (k+1)^{\omega}$, where $[T]$ are the branches of $T$.

Notice that $f^{\ast}(b)$ is also not in any $k$ tree from the ground model. If it were, say in $A$, we will have that $b\in f^{-1}(A)$, but $f^{-1}(A)$ is a $k$ tree from $V$.
\end{proof}

Now, the following results is in the direction of the work done with Noah Schweber:

\begin{lemma}
If a Turing degree has the property that every total function it computes is a branch of a $k+1$ branching computable tree but it computes a function that escapes every $k$ branching computable tree then it computes a function in $(k+1)^{\omega}$ that escapes every $k$ branching tree. 
\end{lemma}

\begin{proof}
The same proof of the above theorem works, since given that $T$ is computable, there is a computable function from it to $(k+1)^{<\omega}$. 
\end{proof}

With this technique, we can answer in a negative way question \ref{other cut}.

\begin{theorem}
The following equality is true $ \mathfrak{v}^{g}_{k}=\max\{\mathfrak{v}^{g}_{k+1}, \mathfrak{L}_{k}\}$. Furthermore, if $\mathfrak{v}^{g}_{k+1}<\mathfrak{v}^{g}_{k}$ then $\mathfrak{L}_{k+1}<\mathfrak{L}_{k}$.
\end{theorem}

\begin{proof}
Let $\mathfrak{v}^{g}_{k+1}=\kappa$. This means that $\omega^{\omega}$ can be covered by $\kappa$ $k+1$-trees.

Now, if $\mathfrak{L}_{k}=\lambda$, this means that $(k+1)^{\omega}$ can be covered by $\lambda$ $k$-trees.

Using a function like the one in the lemma above, given a cover of $\omega^{\omega}$ with $k+1$ trees, we can cover each one of them with $\lambda$ $k$-trees creating a cover of $\kappa \cdot \lambda$ $k$-trees of $\omega^{\omega}$. This means that $\mathfrak{v}^{g}_{k}\leq \kappa \cdot \lambda=\max\{\mathfrak{v}^{g}_{k+1}, \mathfrak{L}_{k}\}$. Since $\mathfrak{v}^{g}_{k}\geq \mathfrak{v}^{g}_{k+1}$ and $\mathfrak{v}^{g}_{k}\geq \mathfrak{L}_{k}$, we have that $\mathfrak{v}^{g}_{k}=\max\{\mathfrak{v}^{g}_{k+1}, \mathfrak{L}_{k}\}$.

For the furthermore, if  $\mathfrak{v}^{g}_{k+1}<\mathfrak{v}^{g}_{k}$ then $\mathfrak{L}_{k}=\mathfrak{v}^{g}_{k}$ and \[\mathfrak{L}_{k+1}\leq \mathfrak{v}^{g}_{k+1} <\mathfrak{v}^{g}_{k}=\mathfrak{L}_{k}\]
\end{proof}

This creates a new question:

\begin{question}
Can $\mathfrak{L}_k$ be express as the maximum or the minimum of other cardinal characteristics?
\end{question}

Finally, the fact that $\mathfrak{v}^{g}_{k}=\max\{\mathfrak{v}^{g}_{k+1}, \mathfrak{L}_{k}\}$ also translate to computability theory, but it becomes a trivial result:

\begin{lemma}
Any Turing degree that computes a function $f\in \omega^{\omega}$ that escapes all $k$ computable trees either computes a function that escapes all $k+1$ computable trees (that same $f$) or a function $g\in (k+1)^{\omega}$ that escapes all $k$-computable trees (the image of $f$ under a certain computable function).
\end{lemma}

\pagebreak

\bibliographystyle{abbrv}
\bibliography{biblio}

\end{document}